\documentclass[10pt]{article}

\usepackage[style=alphabetic,backref,url=false,isbn=false,useprefix=true]{biblatex}
\renewbibmacro{in:}{}
\usepackage{amsfonts}
\usepackage{amsmath}
\usepackage{amsthm}
\usepackage[margin=3cm]{geometry}

\usepackage{setspace}
\usepackage[colorlinks,unicode]{hyperref}

\theoremstyle{plain}
\newtheorem{FactCounter}{dummy}[section]
\newtheorem{Theorem}[FactCounter]{Theorem} %
\newtheorem{Proposition}[FactCounter]{Proposition} %
\newtheorem{Lemma}[FactCounter]{Lemma} %
\newtheorem{Corollary}[FactCounter]{Corollary} %
\theoremstyle{definition}
\newtheorem{Definition}[FactCounter]{Definition} %
\theoremstyle{remark}
\newtheorem*{Remark}{Remark} %
\numberwithin{equation}{section}
\newcommand\N{\mathbb{N}}
\newcommand\Z{\mathbb{Z}}

\newcommand\F{\mathbb{F}}

\let\temp\phi\let\phi\varphi\let \varphi\temp 
\let\temp\theta\let\theta\vartheta\let \vartheta\temp 

\let\0\emptyset

\newcommand\operator[1]{\mathop{\operatorname{#1}}\nolimits}

\newcommand{\Hom}{\operator{Hom}}

\title{Quantum Segre maps via cocycle twists}
\author{Yuri Bazlov and Runyang Chen}
\date{}
\begin{document}
\setcounter{tocdepth}{1}
\maketitle

\begin{abstract}\noindent
A well-known noncommutative deformation $\mathcal A^N_{\mathbf{q}}$ of the polynomial algebra $\mathcal A^N$ can be obtained as a twist of $\mathcal A^N$ by a cocycle on the grading semigroup. 
Of particular interest to us is an interpretation of 
$A^N_{\mathbf{q}}$ as a quantum projective space. 
We outline a general method of cocycle twist quantization
of tensor products and morphisms between algebras graded by monoids
and use it to construct deformations of the classical Segre embeddings of projective spaces.
The noncommutative Segre maps $s_{n,m}$, proposed by 
Arici, Galuppi and Gateva-Ivanova, arise as a particular case of our construction which corresponds to factorizable cocycles in the sense of Yamazaki.
\end{abstract}

\tableofcontents
	
\section{Introduction}

Very generally speaking, 
quantization of an object $X$ refers to some systematical way of replacing the commutative ring of functions $\mathcal O(X)$
by a noncommutative ring $\mathcal O_q(X)$ which depends on a parameter $q$, where one expects to recover the ``classical'' $\mathcal O(X)$
by taking a limit as $q\to 1$. The process may involve quantizing functions from $X$ to $Y$, symmetries or representations of $X$ and so on.
One of the constructions which set the stage for modern quantum algebra is Manin's quantum space $\mathcal O_q(\F^n)$ acted upon by the  ``quantum $\mathrm{GL}_n(\F)$'' \cite{Manin_Koszul}, see also \cite{Manin_CRM}.
The quantization of the affine $n$-space can in fact depend on $\binom n2$ deformation parameters, assembled into a multiplicatively 
antisymmetric $n\times n$ deformation matrix $\mathbf q$.
\medbreak

\noindent Writing the $\mathbf q$-deformation of the polynomial 
ring $\mathcal A^N=\F[x_0,\dots,x_N]$ as $\mathcal A_{\mathbf q}^N$, we view this graded algebra as a quantized homogeneous coordinate ring of a projective space. The algebra $\mathcal A_{\mathbf q}^N$ has many properties expected of a ``quantum $\mathbb P^N$'', surveyed for example in 
\cite{Stafford_vandenBergh} and \cite{Shelton_Tingey}; in particular, it is Artin-Schelter regular and has the correct Hilbert series.
We mention here that $\mathcal A^N_{\mathbf q}$ serves as 
a basis for further quantizations of the projective space in  \cite{MPS}.
\medbreak

\noindent Recently, in \cite{AGGI} Arici, Galuppi and Gateva-Ivanova constructed  a quantization of the Segre embedding 
$\mathbb P^n \times \mathbb P^m \hookrightarrow \mathbb P^N$
where $N = (n+1)(m+1)-1$.  
This quantization is given explicitly as an algebra homomorphism 
$s_{n,m}\colon \mathcal A^N_{\mathbf g} \to \mathcal A^n_{\mathbf q}
\otimes \mathcal A^m_{\mathbf q'}$ where the $\mathbf q$, $\mathbf q'$ are arbitrary deformation matrices and $\mathbf g$ is uniquely determined by $\mathbf q$ and $\mathbf q'$. Here $\otimes$ is the classical tensor product of algebras (where the two factors commute with each other).
\medbreak

\noindent In \cite{ArtinSchelterTate} Artin, Schelter and Tate hold $\mathcal A^N_{\mathbf q}$ as an example of a deformation by a cocycle twist, and show how this leads to a cocycle twist construction 
of the multiparameter quantum $\mathrm{GL}_n$. 
Cocycle deformations of algebras graded by abelian (semi)groups originate from physics and have been used as a method of quantization in various algebraic settings. The twist construction applies more generally to algebras where grading is replaced with a (co)action of a bialgebra, see Majid \cite[\S2.3]{Majid_foundations} and \cite{GiaquintoZhang}. Under some assumptions on the twisting
bialgebra, cocycle twist quantization has been shown to preserve various algebraic properties, for example, being semisimple \cite{torsors} and being Koszul \cite{Jones-Healey}.
\medbreak

\noindent In the present paper, we extend the cocycle twist to morphisms between graded algebras which are compatible with a given morphism between grading monoids. This allows us 
to quantize the classical Segre map, twisting it by a cocycle 
on the monoid $\N^{n+1}\times \N^{m+1}$. 
Up to cohomology, such cocycles are parameterized by triples $(\mathbf q, \mathbf q', \alpha)$ where $\alpha$, represented by a matrix of size 
$(n+1)\times (m+1)$ over $\F^\times$, is a pairing between $\N^{n+1}$ and $\N^{m+1}$, used to deform the tensor product
of the two quantum projective spaces.
\medbreak
 
\noindent We thus obtain noncommutative Segre maps of the form
$$
\mathcal A^N_{\mathbf g} \to \mathcal A^n_{\mathbf q}\otimes_\alpha 
\mathcal A^m_{\mathbf q'},
$$
where $\otimes_\alpha$ is the operation of twisted tensor product of two graded algebras. 
The deformation matrix $\mathbf g$ depends on $\mathbf q$, $\mathbf q'$ and $\alpha$.
There are two extreme cases of this construction:
\begin{itemize}
	\item[(a)] All entries of the matrix $\alpha$ are $1$.
	The resulting maps embed a ``classical product of two quantum projective spaces'' 
	in a quantum projective space, and are exactly the Segre map 
	constructed by Arici, Galuppi and Gateva-Ivanova in \cite{AGGI}.
	\item[(b)] All entries of $\mathbf q$ and $\mathbf q'$ are $1$. 
	We have a new construction realising a ``quantum product of two classical projective spaces'' as a subvariety of a quantum projective space. 
\end{itemize}
\noindent In the present paper, we do not investigate the kernel of 
our general noncommutative Segre map (as done in \cite{AGGI} in case (a)), leaving this to future work.
\medbreak

\subsection*{Organization of the paper}
 
\noindent Twisting an associative $\F$-algebra $A$ graded by a monoid $S$
requires a $2$-cocycle on $S$ with values in the multiplicative group $\F^\times$ of the field $\F$.
In Section~\ref{sect:cohomology}
we collect the facts about $2$-cocycles on monoids that we need.
The machinery used to describe the second cohomology of groups comes from 
Yamazaki \cite{Yamazaki} and Artin, Schelter and Tate \cite{ArtinSchelterTate}. We need a version of this for monoids, so 
we carefully derive the results in a way which avoids taking inverses and group extensions to show that 
$H^2(\N^a, \F^\times)\cong H^2(\Z^a, \F^\times)$ and establish the Yamazaki factorization 
\begin{equation*}
	H^2(\N^a\times \N^b, \F^\times) \cong 
	H^2(\N^a, \F^\times) \times
	H^2(\N^b, \F^\times) \times
	P(\N^a, \N^b, \F^\times),
\end{equation*}
where the last factor is the group of bimultiplicative 
$\F^\times$-valued pairings between the monoids $\N^a$ and $\N^b$.\medbreak

\noindent Section \ref{sect:twists} introduces the general technique of 
twisting an associative algebra by a cocycle on the grading monoid. 
We achieve two principal goals: twisting of a morphism between 
two algebras graded by different monoids (subject to a compatibility condition), and constructing a twisted tensor product of two graded algebras. The final section constructs the quantized Segre maps as an  application of the above technique.

\section{The second cohomology of the free commutative monoid $\mathbb{N}^n$}
\label{sect:cohomology}

\subsection{The second cohomology of a monoid}
Let $S$ be a multiplicative 
monoid with identity $e$.
We let $\Gamma$ be an abelian group, written multiplicatively, 
where $\frac xy$ means $xy^{-1}$, and the identity is denoted by $1$.
The main application later on will be to the case where $\Gamma=\F^\times$ is the multiplicative group of a field.
The action of $S$ on $\Gamma$ will be trivial throughout.

\begin{Definition}
A \textbf{2-cocycle,} or simply a \textbf{cocycle} on a
monoid
$S$ with values in an abelian group $\Gamma$ is a map $\mu\colon S\times S\to \Gamma$ such that
\begin{equation}\label{eq:cocycle}
			\mu(x,yz)\mu(y,z)=\mu(x,y)\mu(xy,z), \quad \mu(x,e)=\mu(e,x)=1
\end{equation}
for all $x,y,z\in S$.
%
For any function $h\colon S \to \Gamma$ with $h(e)=1$, the \textbf{coboundary} $\delta h\colon S \times S \to \Gamma$ of $h$ is
defined by   
$\delta h(x,y) = \dfrac{h(x)h(y)}{h(xy)}$ and is easily seen to be a cocycle.
If $Z^2(S, \Gamma)$ is the abelian group of cocycles under pointwise multiplication, and $B^2(S, \Gamma)$ is the group of coboundaries, 
the \textbf{second cohomology group} of $S$ with coefficients in $\Gamma$ is the factorgroup $H^2(S, \Gamma) = Z^2(S, \Gamma) /B^2(S, \Gamma)$,
whose elements are termed \textbf{cohomology classes.}
Two cocycles are \textbf{cohomologous} if they belong to the same cohomology class.
\end{Definition}
\noindent If $S$ and $T$ are monoids, 
we identify $S$ with the submonoid $S \times \{e_T\}$ of the direct product $S\times T$ via $s \mapsto (s, e_T)$, and likewise $T$ is identified with $\{e_S\}\times T$.
Thus, every cocycle $\mu$ on $S\times T$ restricts to a cocycle on $S$,
written $\mu|_S$. The restriction map
\begin{equation}\label{eq:restriction_on_cocycles}
	Z^2(S\times T, \Gamma) \to Z^2(S, \Gamma) \times Z^2(T, \Gamma),
	\quad
	\mu \mapsto (\mu|_S, \mu|_T), 
\end{equation} 
sends coboundaries to pairs of coboundaries, 
resulting in a well-defined \textbf{restriction homomorphism}
\begin{equation}\label{eq:restriction}
	H^2(S\times T, \Gamma) \to H^2(S, \Gamma) \times H^2(T, \Gamma)
\end{equation} 
on the cohomology. The map \eqref{eq:restriction_on_cocycles}, hence also the map \eqref{eq:restriction}, is surjective: one possible preimage of a pair $(\nu, \xi)$ of cocycles is $\nu\times\xi$ 
defined as follows.  
\begin{Definition}
The \textbf{direct product} of cocycles $\nu\in Z^2(S, \Gamma)$ and 
$\xi\in Z^2(T, \Gamma)$ is the function 
$$
\nu \times \xi \colon (S \times T) \times (S \times T) \to \Gamma,
\quad(\nu\times  \xi) ((s,t), (s',t')) = \nu(s,s')\xi(t,t'),
$$
easily seen to be a cocycle.
A cocycle $\mu \in Z^2(S\times T, \Gamma)$ is \textbf{factorizable}
if $\mu$ is cohomologous to a cocycle of the form $\nu \times \xi$.
\end{Definition}

\subsection{The Yamazaki factorization of a cocycle}

Denote by $P(S, T, \Gamma)$ the set 
of all maps $\alpha\colon S \times T \to \Gamma$ which are 
\textbf{bimultiplicative pairings,} that is,
\begin{equation}\label{eq:pairing}
	\alpha(ss',t)=\alpha(s,t)\alpha(s',t), \quad 
	\alpha(s,tt')=\alpha(s,t)\alpha(s,t')
\end{equation}
for all $s,s'\in S$ and $t,t'\in T$.
Clearly, $P(S, T, \Gamma)$ is a group under pointwise multiplication.
\begin{Proposition}\label{prop:yamazaki_cocycles}
View $S$ and $T$ as submonoids of $S\times T$.
The \textbf{Yamazaki factorization map}
\begin{equation}\label{eq:yamazaki}
	Y\colon Z^2(S\times T, \Gamma) \to  
	Z^2(S, \Gamma) \times Z^2(T, \Gamma) \times P(S, T, \Gamma), 
	\quad 
	Y(\mu) = (\mu|_S, \mu|_T, \alpha_\mu)
\end{equation}
where $\alpha_\mu(s,t) = \dfrac{\mu(s,t)}{\mu(t,s)}$
for $s\in S$ and $t\in T$, is a surjective homomorphism of groups.
\end{Proposition}
\begin{proof}
To check that $\alpha=\alpha_\mu$ is bimultiplicative, write
$\alpha(ss', t) = \mu(ss',t)/\mu(t,ss')$ as  
$$
\frac{\mu(ss',t)\mu(s,s')}{\mu(t,ss')\mu(s,s')} = 
\frac{\mu(s,s't)\mu(s',t)}{\mu(ts,s')\mu(t,s)} =
\frac{\mu(s,s't)\mu(s',t)}{\mu(st,s')\mu(t,s)} =
\frac{\mu(s,s't)\mu(s',t)\mu(s,t)}{\mu(s,ts')\mu(t,s')\mu(t,s)},
$$
by the cocycle equation and the fact that $st=ts$ in $S\times T$. 
But also $ts'=s't$, so after cancellation we get 
$\alpha(s,t)\alpha(s',t)$, proving that $\alpha$ is multiplicative in the left argument. Multiplicativity in the right argument is similar.\medbreak

\noindent It is now clear that $Y$ is a homomorphism. 
Its image contains $Z^2(S,\Gamma)\times Z^2(T,\Gamma) \times \{1\}$ by surjectivity of \eqref{eq:restriction_on_cocycles}. The image of $Y$ also contains $(1,1,\alpha)$ for all $\alpha\in 
P(S, T, \Gamma)$: indeed, define
\begin{equation}\label{eq:cocycle-sigma}
\sigma((s,t), (s',t')) = \alpha(s, t'),	
\end{equation}
and observe that $\sigma$ is bimultiplicative and as such, is easily seen to satisfy cocycle equation \eqref{eq:cocycle}. We have $Y(\sigma) = (1,1,\alpha)$. Surjectivity of $Y$ follows.
\end{proof}
\noindent The map \eqref{eq:yamazaki} was introduced, and the next result proved, in \cite[\S2]{Yamazaki} under the assumption that $S$ and $T$ are groups. See also \cite[Theorem 2.3]{KarpilovskyContinuation}. Here is the monoid version of the Yamazaki theorem \cite[Theoem 2.1]{Yamazaki}:
\begin{Theorem}\label{thm:yamazaki}
Let $S$ and $T$ be monoids and $\Gamma$ be an abelian group.
The Yamazaki factorization map \eqref{eq:yamazaki} induces a bijective homomorphism
\begin{equation}\label{eq:yamazaki_cohomology}
	H^2(S\times T, \Gamma) \xrightarrow{\sim}
	H^2(S, \Gamma)\times H^2(T, \Gamma) \times P(S, T, \Gamma).
\end{equation}	
\end{Theorem}
\begin{proof}
If $\mu$ is a coboundary on $S\times T$, then $\mu|_S$
and $\mu|_T$ are coboundaries, and $\mu$ is symmetric in its arguments so $\alpha_\mu=1$. This shows that \eqref{eq:yamazaki_cohomology} is a well-defined group homomorphism. The map is surjective by Proposition \ref{prop:yamazaki_cocycles}. \medbreak

\noindent To prove injectivity, we assume that a cocycle 
	$\mu$ on $S\times T$ restricts to coboundaries 
	$\delta f$ and $\delta g$ on $S$ and $T$, respectively,
	and $\alpha_\mu=1$. We need to show that $\mu$ is a coboundary. 
Dividing $\mu$ by
$\delta f\times \delta g$, which is a coboundary on $S\times T$, 
we may assume that $\mu|_S=1$ and $\mu|_T=1$.\medbreak
	
\noindent Define the function $h\colon S \times T \to \Gamma$ by $h((s,t)) = 
	1/\mu((s,e_T),(e_S,t))$. Injectivity of \eqref{eq:yamazaki_cohomology} will be proved if we show that $\mu = \delta h$. In the calculation, we write a pair $(s,t)\in S \times T$ as $st$,
	$(s,e_T)$ as $s$ and $(e_S,t)$ as $t$:
$$		\mu(st, s't')  = \frac{ \mu(s, s'tt')  \mu(t,s't') }{ \mu(s,t) } 
		=\frac{\mu(ss', tt') \mu(s,s') \mu(t, s't')}{
			\mu(s,t) \mu(s', tt')}
		= \frac{\mu(ss', tt') \mu(tt', s') \mu(t,t')}{
			\mu(s,t) \mu(s', tt') \mu(s',t')}.
$$			
We have used the cocycle equation and the assumption $\mu(s,s')=1$. 
Yet $\mu(t,t')$ is also $1$, and $\mu(tt', s')$ equals $\mu(s', tt')$ 
because $\alpha_\mu(s',tt')=1$. We obtain $\dfrac{\mu(ss', tt')}{\mu(s,t)\mu(s',t')}$ which is $\delta h((ss', tt'))$, as claimed.
\end{proof}
\noindent We can now obtain a criterion for a cocycle on $S\times T$ to be factorizable.
\begin{Corollary}\label{cor:factorizable}
	A cocycle $\mu\in Z^2(S\times T, \Gamma)$ is factorizable, if, and only if, the third component $\alpha_\mu$ of the Yamazaki factorization $Y(\mu)$ is the trivial pairing $1$.
\end{Corollary}
\begin{proof}
If $\mu$ is factorizable, that is, $\mu$ belongs to the same cohomology
class as $\nu \times \xi$ where $\nu\in Z^2(S, \Gamma)$ and $\xi\in Z^2(T, \Gamma)$, then $\alpha_\mu = \alpha_{\nu \times \xi}$ because $\alpha_\mu$ depends only on the cohomology class of $\mu$. 
If $s\in S$ and $t\in T$, we have $\alpha_{\nu \times \xi}(s, t) = 
\dfrac{ (\nu \times \xi)((s, e_T), (e_S,t)) }
{(\nu \times \xi)((e_S,t), (s, e_T))}$ which is $1$ since $\nu(s, e_S)
	= \nu(e_S, s)=1$ and $\xi(e_T, t)=\xi(t, e_T)=1$.
\medbreak

\noindent Vice versa, if $\alpha_\mu=1$ then $Y(\mu)=Y(\mu|_S \times \mu|_T)$ and so by injectivity of the map \eqref{eq:yamazaki_cohomology} in Theorem \ref{thm:yamazaki}, $\mu$ is cohomologous to $\mu|_S \times \mu|_T$.
\end{proof}

\subsection{Abelian cocycles on a commutative monoid}

We can obtain more information about the 
group $H^2(S, \Gamma)$ in the case where the monoid $S$ is commutative.
%
Define the \textbf{antisymmetrization map} $\beta$ by
\begin{equation*}
	\mu \in Z^2(S, \Gamma) \mapsto \beta_\mu, \quad 
	\beta_\mu(x,y) = \frac{ \mu(x,y) }{ \mu(y,x) }.
\end{equation*}
The same argument as in the proof of Proposition \ref{prop:yamazaki_cocycles} (this time, $st=ts$ for all $s,t\in S$ by commutativity of $S$ and not by the property of a direct product) shows that $\beta$ is a bimultiplicative pairing between $S$ and itself. Moreover, $\beta$ is
\textbf{antisymmetric,} meaning
\begin{equation}\label{eq:antisymmetric}
	\beta_\mu(x,x)=1, \quad \beta_\mu(x,y)\beta_\mu(y,x) = 1.
\end{equation}
\begin{Definition}
If $S$ is an commutative monoid and $\Gamma$ is an abelian group, 
$\mu\in Z^2(S,\Gamma)$ is an \textbf{abelian cocycle} 
if $\mu(x,y)=\mu(y,x)$ for all $x,y\in S$.  \medbreak

\noindent A cohomology class $c$ is abelian if $c$ contains an abelian cocycle. 
Following \cite[\S2.3]{Yamazaki}, we denote the subgroup of \textbf{abelian cohomology classes} in $H^2(S,\Gamma)$ by $H^2_{\text{\rm abel}}(S,\Gamma)$, and the group of \textbf{antisymmetric bimultiplicative pairings} between $S$ and $S$ by 
$P_{\text{\rm a.s.}}(S, \Gamma)$.
\end{Definition}
\noindent A coboundary is clearly an abelian cocycle, hence $c\in H^2_{\text{\rm abel}}(S,\Gamma)$ means that all cocycles in $c$ are abelian. As the antisymmetrization map $\beta$ sends $B^2(S, \Gamma)$ to $\{1\}$, $\beta$ is well defined on $H^2(S, \Gamma)$, and the sequence
\begin{equation}\label{eq:SES}
	H^2_{\text{\rm abel}}(S, \Gamma) \hookrightarrow
	H^2(S, \Gamma) \xrightarrow{\beta} P_{\text{\rm a.s.}}(S, \Gamma)
\end{equation}
is exact in the middle term for any commutative monoid $S$ and abelian group $\Gamma$.	The group $H^2_{\text{\rm abel}}$ is well-behaved with respect to the direct product of monoids:
\begin{Lemma}\label{lem:abelian_direct_product}
If $S$ and $T$ are commutative monoids and $\Gamma$ is an abelian group, the restriction map \eqref{eq:restriction} induces an 
isomorphism
\begin{equation}\label{eq:restriction_abel}
H^2_{\text{\rm abel}}(S\times T, \Gamma) \xrightarrow{\sim}
H^2_{\text{\rm abel}}(S, \Gamma) \times 
H^2_{\text{\rm abel}}(T, \Gamma). 
\end{equation}
\end{Lemma}
\begin{proof}
An abelian cocycle on $S\times T$, when restricted onto $S$ and onto $T$, gives a pair of abelian cocycles, so the map \eqref{eq:restriction_abel} is well defined. 
It is surjective because if
$\nu\in Z^2(S, \Gamma)$ and $\xi\in Z^2(T,\Gamma)$ are both abelian, 
their direct product $\nu\times \xi$ is abelian.
Injectivity follows from Theorem \ref{thm:yamazaki}, taking into account that $\alpha_\mu=1$ if $\mu$ is an abelian cocycle on $S\times T$.
\end{proof}

\subsection{The second cohomology of $\N^n$ and multiplicatively antisymmetric matrices}

Let $\N^n$ denote the free abelian monoid of rank $n$, isomorphic to the direct product of $n$ copies of $\N=\N^1$. In this section, 
we write $\N^n$ multiplicatively, with $g_1,\dots,g_n$ as generators:
\begin{equation*}
	\N^n = \{ g_1^{k_1}g_2^{k_2}\dots g_n^{k_n}: k_1,\dots,k_n\in \Z_{\ge 0}\}.
\end{equation*}
Our aim is to obtain an explicit description of the second cohomology group $H^2(\N^n, \Gamma)$ where $\Gamma$ is an abelian group. We begin with
\begin{Lemma}
$H^2(\N^1,\Gamma)=\{1\}$ and 
$H^2_{\text{\rm abel}}(\N^n,\Gamma)=\{1\}$ for all $n\ge 1$ and for all 
abelian groups $\Gamma$.
\end{Lemma}
\begin{proof}
Let $\mu\in Z^2(\N^1, \Gamma)$ and let $g$ be the generator of $\N^1$. 
Define $h\colon \N^1 \to \Gamma$ by the rules $h(e)=h(g)=1$ 
and $\mu(g, g^p)= h(g^p)/h(g^{p+1})$ for $p=1,2,\dots$
Then the equation 
\begin{equation}\label{eq:mu-f}
	\mu(g^p, g^q) = \frac{h(g^p) h(g^q)}{h(g^{p+q})} 
\end{equation}
holds  in the case $p=1$ for all $q\ge 1$.  
The cocycle equation affords the inductive step from $p$ to $p+1$: 
$$
\mu(g^{p+1}, g^q) = \mu(g^p, g^{q+1})\frac{\mu(g, g^q)}{\mu(g^p, g)}
= \frac{h(g^p) h(g^{q+1})}{h(g^{p+q+1})} \frac{1\cdot h(g^{q})}{h(g^{q+1})} \frac{h(g^{p+1})}{h(g^p)\cdot 1}  
= \frac{h(g^{p+1}) h(g^q)}{h(g^{p+q+1})},
$$
proving that \eqref{eq:mu-f} holds for all $p\ge 1$, so that 
$\mu$ is a coboundary.\medbreak

\noindent We have $H^2_{\text{\rm abel}}(\N^1,\Gamma) \subseteq 
H^2(\N^1,\Gamma)=\{1\}$, so 
by Lemma~\ref{lem:abelian_direct_product}, 
$H^2_{\text{\rm abel}}(\N^n,\Gamma)=\{1\}$ for all~$n$. 
\end{proof}

\begin{Proposition}\label{prop:as-pairings}
	For any abelian group $\Gamma$, 
	the map $\beta\colon H^2(\N^n, \Gamma) \to P_{\text{\rm a.s.}}(\N^n, \Gamma)$, given by \eqref{eq:SES}, is an isomorphism.
\end{Proposition}
\begin{proof}
Lemma~\ref{lem:abelian_direct_product} and \eqref{eq:SES} imply that $\beta$ has trivial kernel and is injective.
Now, if $b$ is a $\Gamma$-valued antisymmetric bimultiplicative pairing on $\N^n$, define $\mu\colon \N^n \times \N^n\to \Gamma$ on the generators $g_1,\dots, g_n$ of $\N^n$ by 
$\mu(g_i, g_j) = b(g_i, g_j)$ if $i<j$, $\mu(g_i, g_j) = 1$ if $i\ge j$. Extend $\mu$ to be a bimultiplicative pairing; $\mu$ may not be
antisymmetric, but bimultiplicativity implies that $\mu$ is a cocycle. Clearly $\beta_\mu = b$, proving surjectivity of $\beta$. 
\end{proof}
\noindent Since any bimultiplicative pairing $b$ on $\N^n$ is uniquely determined by the values $q_{ij} := b(g_i, g_j)$, 
Proposition~\ref{prop:as-pairings} and its proof lead us to the following
\begin{Theorem}\label{thm:antisym-matrices}
Let $M_{\text{\rm a.s.}}^n(\Gamma)$ denote the set of $n\times n$
matrices $\mathbf{q}$ with entries in $\Gamma$ which are \textbf{multiplicatively antisymmetric,} meaning that 
$q_{ii}=1$ and $q_{ij}q_{ji}=1$ for all $i,j=1,\dots,n$.
There is a bijection 
$$
M_{\text{\rm a.s.}}^n(\Gamma) \xrightarrow{\sim} H^2(\N^n, \Gamma), 
$$
sending a matrix $\mathbf q=(q_{ij})_{i,j=1}^n$ to the 
cohomology class which contains a unique cocycle $\mu_{\mathbf q}$ such that
$$
\mu_{\mathbf q}(g_i, g_j) = \begin{cases} 
	q_{ij} &\text{if }i<j, \\
	1 &\text{if }i\ge j,
	\end{cases}
	\qquad\mu_{\mathbf q}\text{ is bimultiplicative. \qed} 
$$
\end{Theorem}
\noindent We now interpret the Yamazaki bijection  \eqref{eq:yamazaki_cohomology} for the direct product $\N^a \times\N^b$ in terms of matrices.\medbreak

\noindent Let $\Gamma$ be an abelian group and let $a,b\ge 1$. Recall the
Yamazaki map $\mu \mapsto (\mu|_{\N^a}, \mu|_{\N^b}, \alpha_\mu)$ which by Theorem \ref{thm:yamazaki} induces an isomorphism on cohomology. In terms of matrices, this map corresponds to writing a multiplicatively antisymmetric
matrix $\mathbf{q}$ of size $(a+b)\times (a+b)$ with entries in $\Gamma$ as a triple of matrices:
$$
\mathbf{q} \quad \mapsto \quad ((q_{ij})_{1\le i,j\le a},\ (q_{ij})_{a+1\le i,j\le a+b},
(q_{ij})_{1\le i\le a,\ a+1\le j\le a+b}),
$$
of size $a\times a$, $b\times b$ and $a\times b$, respectively,
where the first two matrices are multiplicatively antisymmetric,
and the third matrix is arbitrary.
Note:
\begin{itemize}
	\item the set of all $a\times b$ matrices with entries in $\Gamma$ is in bijection with $P(\N^a, \N^b, \Gamma)$, because a pairing 
	is defined arbitrarily on generators;
	\item multiplicatively antisymmetric matrices $\mathbf q$ are in one-to-one correspondence with triples of matrices with the given properties.
\end{itemize}
\noindent Corollary \ref{cor:factorizable} implies the following criterion.\medbreak

\begin{Corollary}
A cocycle $\mu$ on $\N^a \times \N^b$ is cohomologous to a direct product of a cocycle on $\N^a$ and a cocycle on $\N^b$ if, and only if, 
the pairing $\alpha_\mu$ between $\N^a$ and $\N^b$ is identically $1$.
Equivalently, $\mu$ corresponds to a multiplicatively antisymmetric matrix $\mathbf q$ with the upper right $a\times b$ block filled by $1$: $q_{ij}=1$ whenever $1\le i\le a$ and $a+1\le j\le a+b$.
\end{Corollary}
	
\section{Cocycle twists of graded algebras}
\label{sect:twists}	

\subsection{A cocycle twist of a graded associative algebra}

Let $\F$ be a field and $S$ be a monoid. The category of $S$-graded $\F$-algebras consists of unital associative algebras $A$ such that 
$A = \bigoplus_{s\in S} A_s$ with $A_s A_t \subseteq A_{st}$ for all $s,t\in S$, and the identity element $1_A$ of $A$ lies in $A_e$ 
where $e$ is the identity of $S$. \medbreak

\begin{Definition}\label{def:twist}
\cite[Definition I.12.16]{BrownGoodearl}
		Let $S$ be a monoid and $A$ an $S$-graded $\F$-algebra. Fix a 2-cocyle $\mu\in Z^2(S, \F^{\times})$. Let $A'$ be a copy of $A$, viewed as $S$-graded vector space over $\F$ having natural $S$-graded vector space isomorphism $a\mapsto a'$. Define a product on $A'$ by
		\begin{equation}\label{eq:twisted_product}
		a'*b'=\mu(x,y)(ab)'
	    \end{equation}
	for homogeneous elements $a\in A_x$ and $b\in A_y$, and extend by linearity, where $ab$ is the product of $a$ and $b$ in $A$. Then $A'$ is an $S$-graded $\F$-algebra, called the \textbf{twist} of $A$ by $\mu$; the map $a\mapsto a'$ is called the \textbf{twist map}. 
\end{Definition}	
\noindent We denote the twist of $A$ by $\mu$ as $A_\mu$ and think of $A_\mu$ as having the same elements as $A$ but a different product.
		We write the product on $A_\mu$ as $*$, or for emphasis 
		as $*_\mu$.
Note that cocycle equation \eqref{eq:cocycle} guarantees associativity and unitality of $A_\mu$, where the identity element is $1'_A$. Moreover, $A_\mu$ remains an $S$-graded algebra. 
We now note that the isomorphism class of the algebra $A_\mu$ 
depends only on the cohomology class of $\mu$ in $H^2(S,\F^\times)$:\medbreak

\begin{Lemma}\label{lem:coboundary_isomorphism}
Let $A$ be an $S$-graded algebra, $\mu, \nu \in Z^2(S, \F^\times)$, and consider the twists $A_\mu$ (with canonical graded space isomorphism $a\in A \mapsto a'\in A_\mu$) and $A_\nu$ (with $a\in A \mapsto a''\in A_\nu$). 
If the cocycles $\mu,\nu$ are cohomologous, that is $\mu = (\delta h)\nu$ with $h\colon S \to \F^\times$,
 then there is an isomorphism $\phi_h\colon A_\mu \to A_\nu$ of $S$-graded algebras, 
 given on homogeneous $a\in A_s$ by $\phi_h(a') = h(s)a''$.
\end{Lemma}
\begin{proof} It is clear that $\phi_h$ is an isomorphism of graded vector spaces as $h(s)$ are invertible elements of $\F$.
For $a\in A_s$, $b\in A_t$
we have $\phi_h(a'*_\mu b') = \mu(s,t) \phi_h((ab)') = \mu(s,t)h(st) (ab)''$.
On the other hand, $\phi_h(a')*_\nu\phi_h(b') = h(s)a''*_\nu h(t)b'' = \nu(s,t) h(s)h(t) (ab)''$. This is equal to $\phi_h(a'*_\mu b')$ as $\mu = (\delta h)\nu$.     
\end{proof}
\noindent It turns out that the morphisms between twists of two $S$-graded algebras by $\mu$ are the same, as a set, as morphisms between the untwisted algebras: 
	\begin{Lemma}\label{lem:morphisms_same_monoid}
		Let $A$, $B$ be $\F$-algebras graded by a monoid $S$, and let  $\phi\colon A\to B$ be a morphism of $S$-graded algebras, that is to say $\phi(A_s)\subseteq B_s$ for all $s\in S$. Assume $\mu\in Z^2(S,\F^{\times})$. Then the map $\phi'\colon A_\mu \to B_\mu$, defined by $\phi'(a')=(\phi(a))'$, is also a morphism of algebras.
	\end{Lemma}
	\begin{proof} 
Consider $S$-homogeneous elements $a_1\in A_{s_1}$, $a_2\in A_{s_2}$ with $s_1, s_2\in S$. One has 
$\phi'(a_1'* a_2') = \phi'(\mu(s_1,s_2) (a_1a_2)') = \mu(s_1,s_2) \phi(a_1a_2) '$. Since $\phi$ intertwines the untwisted products on $A$ and $B$,  this is  $\mu(s_1,s_2)\left( \phi(a_1)\phi(a_2) \right)' = 
\left(\phi(a_1)\right)' * \left(\phi(a_2)\right)' =
\phi'(a_1') * \phi'(a_2')$
because $\phi'(a_1')\in B_{s_1}$ and $\phi'(a_2')\in B_{s_2}$.
We have shown that $\phi'(a_1'* a_2') = \phi'(a_1') * \phi'(a_2')$
for homogeneous $a_1$ and $a_2$; 
by linearity, the same is true for arbitrary $a_1,a_2\in A$.
\end{proof}
	
\subsection{Morphisms between twists of algebras graded by different monoids}

We extend Lemma~\ref{lem:morphisms_same_monoid} to the case where each one of the algebras $A$ and $B$ has its own grading monoid. There must be given a morphism between the grading monoids, and the twisting cocycle transforms as follows:\medbreak

\begin{Definition}
Let $S\xrightarrow{f} T$ be a morphism between monoids $S$ and $T$.
If $\mu\colon T \times T \to \Gamma$ is a cocycle on $T$, the \textbf{pullback} of $\mu$ along $f$ is 
$\mu^f\colon S \times S \to \Gamma$ defined by $\mu^f(s,s')
=\mu(f(s), f(s'))$. 
\end{Definition}
\noindent It is easy to see that $\mu^f$ is a cocycle on $S$. 
The following result is proved by inserting $f$ in the appropriate places in the proof of Lemma~\ref{lem:morphisms_same_monoid}:
\begin{Lemma}\label{general case of lemma}
	Let $S\xrightarrow{f} T$ be a mophism of monoids, $A$ be an $S$-graded $\F$-algebra and $B$ be a $T$-graded $\F$-algebra. 
	Suppose that $\phi\colon A\to B$ is a morphism of graded algebras 
	compatible with $f$, that is, 
	\begin{equation*}
		\phi(A_s)\subseteq B_{f(s)}
	\end{equation*}
	for all $s\in S$. Then for any cocycle $\mu\in Z^2(T,\F^\times)$, the map 
$$
\phi'\colon A_{\mu^f}\rightarrow B_{\mu},
$$
which coincides with $\phi$ on the underlying vector spaces of $A_{\mu^f}$ and $B_{\mu}$, is a morphism of algebras.\qed
\end{Lemma}

\medbreak

\subsection{Twisted tensor product of graded algebras}	
\label{subsect:twisted_tensor_product}

Consider two monoids $S$ and $T$. Let $B$ be an $S$-graded $\F$-algebra and $C$ be a $T$-graded $\F$-algebra. The space $B\otimes C$, 
with the classical multiplication given by 
$(b\otimes c)(b'\otimes c')=bb'\otimes cc'$ for all $b,b'\in B$ and 
$c,c'\in C$ and the identity $1_{B\otimes C} = 1_B\otimes 1_C$, 
is a unital associative algebra which is $S\times T$-graded via
\begin{equation*}
		B\otimes C=\bigoplus_{(s,t)\in S\times T} (B\otimes C)_{(s,t)},
\end{equation*}
where $(B\otimes C)_{(s,t)}=B_s\otimes C_t$. 
By identifying $B\cong B\otimes\{1_C\}$ and $C \cong \{1_B\}\otimes C$, 
we can consider $B$ and $C$ as subalgebras of $B\otimes C$.
This \textbf{classical tensor product} of two graded algebras is the simplest example of the following graded version of \cite[Definition 21.3]{Majid_primer}:
\begin{Definition}
A \textbf{graded algebra factorization} is an graded algebra  
$X$ and two graded subalgebras $B,C\subseteq X$ such that the product map $X\otimes X \to X$, restricted onto $B\otimes C$,   
gives an isomorphism $B\otimes C\xrightarrow{\sim} X$ of vector spaces.
\end{Definition}	
\noindent We now introduce a family of graded algebra factorizations into $B$ and $C$ parameterized by pairings $\alpha\in P(S, T, \F^\times)$. The case $\alpha=1$ corresponds to the classical tensor product $B\otimes C$:
\begin{Lemma}\label{lem:twisted_tensor_product}
Let $B$ be an $S$-graded $\F$-algebra, $C$ be a $T$-graded $\F$-algebra, and $\alpha \in P(S, T, \F^\times)$ be a pairing. 
There exists unique $S\times T$-graded algebra product $*$  
on the underlying vector space $B\otimes C$, 
such that $B\cong B\times \{1_C\}$ and $C\cong \{1_B\}\times C$ are 
subalgebras, and $c*b = \alpha(s,t)b\otimes c$ for all $b\in B_s$
and $c\in C_t$ with $s\in S$, $t\in T$.	
\end{Lemma}
\begin{proof}
Clearly, if the associative product $*$ exists, it is unique as
$(b\otimes c)*(b'\otimes c')$ must be given by 
$\alpha(s,t) bb'\otimes cc'$ whenever $b'\in B_s$ and $c\in C_t$.\medbreak

\noindent To show existence of $*$, define, in a slight modification of \eqref{eq:cocycle-sigma}, the cocycle $\tau \in Z^2(S\times T, \F^\times)$ 
by 
$$
\tau((s,t), (s',t')) = \alpha(s',t).
$$
Consider the twist $(B\otimes C)_\tau$ of the classical tensor product algerba $B\otimes C$. The twisted product $*=*_\tau$ on the underlying vector space $B\otimes C$ has the 
required properties; in particular, \eqref{eq:twisted_product}
implies that $c*b = \tau((e_S,t),(s,e_S))b\otimes c = \alpha(s,t)b\otimes c$ for $b\in B_s$ and $c\in C_t$, as required.
\end{proof}
\begin{Remark}
The cocycle $\tau\in Z^2(S\times T, \F^\times)$ used in the proof of the Lemma has Yamazaki factorization $Y(\tau) = (1, 1, \frac1\alpha)$ in $Z^2(S,\F^\times)
\times Z^2(T, \F^\times) \times P(S, T, \F^\times)$. 
The reason for inverting $\alpha$ is purely technical, namely,
$\alpha$ appears in the equation for $c*b$ and not $b*c$.
\end{Remark}
\begin{Definition}
Given an $S$-graded algebra $B$, a $T$-graded algebra $C$ and a pairing 
$\alpha\in P(S, T, \F^\times)$, the algebra factorization described in 
Lemma \ref{lem:twisted_tensor_product} is called the \textbf{$\alpha$-twisted tensor product} of $B$ and $C$ and 
denoted $B\otimes_\alpha C$. 
\end{Definition}
\noindent It turns out that all twists of the classical tensor product $B\otimes C$ by cocycles
on $S\times T$ can be written as twisted tensor products of a twist of $B$ and a twist of $C$, as follows.\medbreak

\begin{Proposition}
Let $B$ be an $S$-graded $\F$-algebra, $C$ be a $T$-graded $\F$-algebra, and $\mu\in Z^2(S\times T, \F^\times)$. 
Suppose that the Yamazaki factorization \eqref{eq:yamazaki} of $\mu$
is $Y(\mu)=(\nu, \xi, 1/\alpha)$ where $\nu \in Z^2(S, \F^\times)$, 
$\xi\in Z^2(T, \F^\times)$ and $\alpha\in P(S, T, \F^\times)$. 
Then the 
twisted algebra $(B\otimes C)_\mu$ is isomorphic to the 
twisted tensor product
$B_\nu \otimes_\alpha C_\xi$.
\end{Proposition}
\begin{proof}
By Theorem \ref{thm:yamazaki}, $\mu$ is cohomologous to the cocycle 
$\mu_1 = (\nu \times \xi)\tau$ where $Y(\tau) = (1,1,1/\alpha)$,
and so by Lemma \ref{lem:coboundary_isomorphism}, 
$(B\otimes C)_\mu$ is isomorphic to $(B\otimes C)_{\mu_1}$.
It is easy to observe that  $(B\otimes C)_{\nu \times \xi}$
is the classical tensor product $B_\nu \otimes C_\xi$, and so
$(B\otimes C)_{(\nu \times \xi)\tau}\cong \left(B_\nu \otimes C_\xi\right)_\tau$ which is $B_\nu \otimes_\alpha  C_\xi$, exactly as in the proof of Lemma \ref{lem:twisted_tensor_product}. 
\end{proof}	
\begin{Remark}
	In particular, twists of $B\otimes C$ by factorizable cocycles on $S\times T$ are classical tensor products 
	$B_\nu \otimes C_\xi$ of twists of $B$ and $C$.
\end{Remark}

\section{Quantum Segre maps via cocycle quantization}

\subsection{Classical Segre embeddings}
In classical algebraic geometry, 
the polynomial algebra $\mathcal A^n := \F[x_0,\dots,x_n]$ serves as the homogeneous coordinate ring of the $n$-dimensional projective space $\mathbb P^n := \mathbb{P}^n(\F)$. Projective subvarieties of $\mathbb{P}^n$
are zero loci of homogeneous ideals of $\mathcal A^n$, with the exception of $\langle x_0, \dots, x_n\rangle$. 
Given $n,m\ge 1$, the \textbf{classical Segre map}
is an embedding 
$$
\mathbb P^n \times \mathbb P^m \hookrightarrow \mathbb P^{(n+1)(m+1)-1}
$$
which equips the Cartesian product of $\mathbb P^n$ and $\mathbb P^m$
with a structure of a projective subvariety of $\mathbb P^{(n+1)(m+1)-1}$.
To describe this map explicitly, we assume that 
the $(n+1)(m+1)$ homogeneous coordinates on $\mathbb P^{(n+1)(m+1)-1}$
are labeled by \textbf{pairs of indices:}
$$
z_{00}, z_{01}, \dots, z_{nm}.
$$
If $[x_0:x_1:\dots:x_n]$ are homogeneous coordinates of a point $x$ on $\mathbb P^n$, and  $[y_0:y_1:\dots:y_m]$, of a point $y$ on $\mathbb P^m$, the image of $(x, y)$ under the Segre embedding is the point $z$ with homogeneous coordinates $z_{ij} = x_i y_j$, $0\le i\le n$, $0\le j\le m$.
All such point $z$ form the \textbf{Segre subvariety} 
of $\mathbb P^{(n+1)(m+1)-1}$.\medbreak

\noindent In order to quantize the Segre embedding, we need to consider its equivalent form as a morphism between homogeneous coordinate rings. 
(We will later use gradings on those rings 
to quantize them into noncommutative algebras.)
The above coordinate description means that the 
classical Segre map is the following morphism of commutative algebras:
	\begin{equation*}
		s_{n,m}\colon \F[z_{00},\dots,z_{nm}]\rightarrow\F[x_0,\dots,x_n]\otimes
		\F[y_0,\dots,y_m], 
\qquad 		s_{n,m}(z_{ij})=x_i\otimes y_j.
	\end{equation*}
The image of $s_{n,m}$ is the \textbf{Segre product} 
\cite[Ch.3, \S2]{PP_QA}
of $\mathcal A^n$ and $\mathcal A^m$, namely the algebra spanned by monomials where the $x$-degree and the $y$-degree are the same.
The kernel of $s_{n,m}$ is the ideal generated by relations 
which say that the matrix $(z_{ij})$ of homogeneous coordinates 
has rank $1$: that is, all $2\times 2$ minors of this matrix are zero.
That these relations are quadratic justifies the alternative name \textbf{Segre quadric} for the Segre subvariety of $\mathbb P^{(n+1)(m+1)-1}$.\medbreak

\subsection{Gradings on the homogeneous coordinate rings}
	
We view $\F[z_{00},\dots,z_{nm}]$ as an algebra which is graded by the monoid $\N^{(n+1)(m+1)}$.
For convenience we now write this monoid additively.  
Elements of $\N^{(n+1)(m+1)}$ are viewed as matrices of size 
$(n+1)\times (m+1)$, with rows indexed by $0,\dots,n$ and columns indexed by $0, \dots, m$. 
The degree of the generator $z_{ij}$ is $e_{ij}\in\N^{(n+1)(m+1)}$,  the unit matrix whose non-zero entry is in the $(i,j)$ position. \medbreak

\noindent The polynomial algebra $\F[x_0,\dots,x_n]$ is graded by the monoid 
$\N^{n+1}$. Hence 
the tensor product $\F[x_0,\dots,x_n]\otimes\F[y_0,\dots,y_m]$ 
acquires, as per  Section \ref{subsect:twisted_tensor_product} above, grading by the monoid $\N^{n+1}\times\N^{m+1}$
(not the same as $\N^{(n+1)(m+1)}$). 
Elements of $\N^{n+1}\times \N^{m+1}$ are viewed as pairs of vectors of size $(n+1, m+1)$.
We write $\alpha_i\in\N^{n+1}$, $\beta_j\in\N^{m+1}$ for the unit vectors whose non-zero entries are in the $i$th position and in the $j$th position, respectively.
Here $0\le i\le n$, $0\le j\le m$.	\medbreak

\noindent To consider the classical Segre map $s_{n,m}$ as a map between graded algebras 
with different grading monoids, we define the morphism
\begin{equation}\label{eq:morphism-f}
	f\colon \N^{(n+1)(m+1)}\rightarrow\N^{n+1}\times\N^{m+1}, 
	\quad f(e_{ij})=(\alpha_i,\beta_j).
\end{equation} 
Since $\N^{(n+1)(m+1)}$ is a free commutative monoid with generators $\{e_{ij}:i=0,\dots,n, \ j=0,\dots,m\}$, any mapping of each generator $e_{ij}$ to an element of $\N^{n+1}\times\N^{m+1}$ uniquely extends $f$ 
to $\N^{(n+1)(m+1)}$ as a morphism of monoids. The following is an easy observation:\medbreak

\begin{Lemma}\label{lem:compatible}
The classical Segre map $s_{n,m}$ between the $\N^{(n+1)(m+1)}$-graded $\F$-algebra $\F[z_{00},\dots,z_{nm}]$ and the 
$\N^{n+1}\times\N^{m+1}$-graded $\F$-algebra 
$\F[x_0,\dots,x_n]\otimes\F[y_0,\dots,y_m]$ is compatible with the morphism $f$ of monoids, in the sense of Lemma \ref{general case of lemma}. \qed 
\end{Lemma}

\subsection{Quantum projective spaces as cocycle twists}

The polynomial algebra
$$
\mathcal A^N = \F[x_0,\dots,x_N]
$$
in $N+1$ variables over the field $\F$ has a well-studied family of noncommutative deformations, as follows. 
\begin{Definition}
Let $\mathbf q = (q_{ij})_{i,j=0}^N$ be a
matrix of size $(N+1)\times (N+1)$ with entries in $\F^\times$ which is
multiplicatively antisymmetric (see Theorem \ref{thm:antisym-matrices}). The $\F$-algebra $\mathcal A^N_{\mathbf q}$, generated by $X_0,\dots,X_N$ subject to the relations 
\begin{equation}\label{eq:q-commute}
	X_j X_i - q_{ji} X_i X_j =0, \quad 0\le i<j\le N,
\end{equation}
is the \textbf{quantum projective space} defined by $\mathbf q$.
\end{Definition}
\noindent Let $M_{\text{\rm a.s.}}^{N+1}(\F^\times)$ be the set 
of $(N+1)\times (N+1)$ multiplicatively 
antisymmetric matrices. 
Recall from Theorem \ref{thm:antisym-matrices} the  bijection
$$
M_{\text{\rm a.s.}}^{N+1}(\F^\times) \to  H^2(\N^{N+1}, \F^\times), 
\quad \mathbf q \mapsto \mu_{\mathbf q}.
$$
This bijection has explicit manifestation in the following:
\begin{Proposition}\label{prop:twist}
The algebra $\mathcal A^N_{\mathbf q}$ is isomorphic to the 
cocycle twist $\left(\mathcal A^N\right)_{\mu_{\mathbf q}}$. 
\end{Proposition}
\begin{proof}
As in Definition \ref{def:twist}, 
we write the elements of $\left(\mathcal A^N\right)_{\mu_{\mathbf q}}$
as $p'$ where $p\in \mathcal A^N$.
We claim that there exists an algebra homomorphism 
$\phi\colon \mathcal A^N_{\mathbf q} \to \left(\mathcal A^N\right)_{\mu_{\mathbf q}}$ given on the generators of $\mathcal A^N_{\mathbf q}$ by $\phi(X_i) = x_i'$. We verify 
relations \eqref{eq:q-commute}: if $0\le i<j\le N$, 
\begin{align*}
\phi(X_j)*\phi(X_i) - q_{ji} \,\phi(X_i)*\phi(X_j) & = 
x_j' * x_i' - q_{ji}\, x_i' * x_j' 
\\ & = \mu_{\mathbf q}(g_j, g_i) (x_jx_i)'
- q_{ji} \,\mu_{\mathbf q}(g_i, g_j) (x_ix_j)'
\\ & =  1(x_jx_i)' - q_{ji} q_{ij} (x_i x_j)' = 0
\end{align*}
as $x_jx_i=x_ix_j$ in $\mathcal A^N$ and $q_{ji}q_{ij}=1$.
(We have used the formula for $\mu_{\mathbf q}$ given in Theorem \ref{thm:antisym-matrices}.)\medbreak

\noindent This shows that $\phi$ is a well-defined homomorphism. 
Relations \eqref{eq:q-commute} are easily seen to imply that standard monomials 
$X_0^{k_0}\dots X_N^{k_N}$, $k_0,\dots,k_N\in \Z_{\ge 0}$, form a spanning set in 
$\mathcal A^N_{\mathbf q}$; the homomorphism $\phi$ sends  
$X_0^{k_0}\dots X_N^{k_N}$ to a non-zero scalar multiple 
of $(x_0^{k_0}\dots x_N^{k_N})'$ in $ \left(\mathcal A^N\right)_{\mu_{\mathbf q}}$, where by definition 
these monomials form a basis. A spanning set carried by a linear map to 
a basis must itself be a basis, and so $\phi$ carries a basis to a basis hence is bijective.
\end{proof}

\begin{Remark}
It is well known that standard monomials form a basis of $\mathcal A_{\mathbf q}$. This can be shown, for example, by an easy application of the Diamond Lemma \cite{Bergman}. 
Proposition \ref{prop:twist} offers yet another way to formalize the proof of this fact.  
\end{Remark}

\subsection{Quantizations of the classical Segre map. The factorizable cocycle case}

From now on, we fix $n\ge 1$ and $m\ge 1$ and write 
\begin{itemize}
	\item $\mathcal A^n$ to denote the polynomial algebra 
	$\F[x_0,\dots,x_n]$, $\mathcal A^m$ to denote $\F[y_0,\dots,y_m]$, and 
	\item $\mathcal A^{(n+1)(m+1)-1}$ to denote 
	$\F[z_{00},\dots,z_{nm}]$.
\end{itemize}
Recall the morphism 
$$
f\colon \N^{(n+1)(m+1)}\to \N^{n+1}\times \N^{m+1}
$$ 
of monoids, defined in \eqref{eq:morphism-f}. 
These are grading monoids for $\mathcal A^{(n+1)(m+1)-1}$ and 
for $\mathcal A^n \otimes \mathcal A^m$, respectively.
The classical Segre map $s_{n,m}\colon \mathcal A^{(n+1)(m+1)-1} \to 
\mathcal A^n \otimes \mathcal A^m$ is compatible with $f$ (Lemma \ref{lem:compatible}), which allows us to make the following definition.
\begin{Definition}
Let $\mu$ be a cocycle 	in $Z^2(\N^{n+1}\times \N^{m+1}, \F^\times)$.
The \textbf{quantum Segre map} 
	\begin{equation*}
	\left(s_{n,m}\right)_{\mu}\colon 
	\left( \mathcal A^{(n+1)(m+1)-1}\right)_{\mu^f}\to 
	\left( \mathcal A^n \otimes \mathcal A^m\right)_{\mu}.
\end{equation*}
is the morphism of algebras given by Lemma \ref{general case of lemma};
that is, the unique morphism 
which is defined on generators $z_{ij}'$ of $\left( \mathcal A^{(n+1)(m+1)-1}\right)_{\mu^f}$ by  
 $s_{n,m}(z_{ij}')=(x_i\otimes y_j)'$.
\end{Definition} 
\noindent We finish by analyzing the case where the cocycle $\mu$ on $\N^{n+1}\times \N^{m+1}$ is factorizable. The next Proposition follows from our results so far.
\begin{Proposition}
Suppose that the cocycle $\mu$ on $\N^{n+1}\times \N^{m+1}$ is factorizable of the form $\nu\times \xi$, where 
$\nu$ is a cocycle on $\N^{n+1}$ and $\xi$ is a cocycle on $\N^{m+1}$. Let $\mathbf q$ and $\mathbf q'$ be multiplicatively antisymmetric matrices
which correspond to $\nu$ and $\xi$, respectively, via Theorem \ref{thm:antisym-matrices}. Then:
\begin{enumerate}
	\item the cocycle $\mu^f$ corresponds to the multiplicatively antisymmetric matrix $\mathbf g$ which is the Kronecker product of
	$\mathbf q$ and $\mathbf q'$;
	\item the twist $\left(\mathcal A^n \otimes \mathcal A^m\right)_\mu$ 
	is the classical tensor product $\mathcal A^n_{\mathbf q} \otimes 
	\mathcal A^m_{\mathbf q'}$ of two quantum projective spaces;
	\item the map $\left(s_{n,m}\right)_{\mu^f}$ coincides with the
	noncommutative Segre map constructed in \rm \cite{AGGI}. \qed
\end{enumerate}
\end{Proposition}

\printbibliography

@book {BrownGoodearl,
	AUTHOR = {Brown, Ken A. and Goodearl, Ken R.},
	TITLE = {Lectures on algebraic quantum groups},
	SERIES = {Advanced Courses in Mathematics. CRM Barcelona},
	PUBLISHER = {Birkh\"auser Verlag, Basel},
	YEAR = {2002},
	PAGES = {x+348},
	ISBN = {3-7643-6714-8},
	MRCLASS = {16W35 (17B37 20G42 81R50)},
	MRNUMBER = {1898492},
	MRREVIEWER = {M\'aty\'as\ Domokos},
	DOI = {10.1007/978-3-0348-8205-7},
	zzURL = {https://doi.org/10.1007/978-3-0348-8205-7},
}

@article {ArtinSchelterTate,
	AUTHOR = {Artin, Michael and Schelter, William and Tate, John},
	TITLE = {Quantum deformations of {${\rm GL}_n$}},
	JOURNAL = {Comm. Pure Appl. Math.},
	FJOURNAL = {Communications on Pure and Applied Mathematics},
	VOLUME = {44},
	YEAR = {1991},
	NUMBER = {8-9},
	PAGES = {879--895},
	ISSN = {0010-3640,1097-0312},
	MRCLASS = {17B37 (14A22 16W30)},
	MRNUMBER = {1127037},
	MRREVIEWER = {S.\ Paul\ Smith},
	DOI = {10.1002/cpa.3160440804},
	zzURL = {https://doi.org/10.1002/cpa.3160440804},
}

@article {AGGI,
	AUTHOR = {Arici, Francesca and Galuppi, Francesco and Gateva-Ivanova,
	Tatiana},
	TITLE = {Veronese and {S}egre morphisms between non-commutative
	projective spaces},
	JOURNAL = {Eur. J. Math.},
	FJOURNAL = {European Journal of Mathematics},
	VOLUME = {8},
	YEAR = {2022},
	PAGES = {S235--S273},
	ISSN = {2199-675X,2199-6768},
	MRCLASS = {16S37 (16S10 16S15 16S38 81R60)},
	MRNUMBER = {4452844},
	MRREVIEWER = {Padmini\ Veerapen},
	DOI = {10.1007/s40879-022-00547-3},
	zzURL = {https://doi.org/10.1007/s40879-022-00547-3},
}

@article {Yamazaki,
	AUTHOR = {Yamazaki, Keijiro},
	TITLE = {On projective representations and ring extensions of finite
	groups},
	JOURNAL = {J. Fac. Sci. Univ. Tokyo Sect. I},
	FJOURNAL = {Journal of the Faculty of Science. University of Tokyo.
	Section I},
	VOLUME = {10},
	YEAR = {1964},
	PAGES = {147--195},
	ISSN = {0368-2269},
	MRCLASS = {20.80},
	MRNUMBER = {180608},
	MRREVIEWER = {Michael\ I.\ Rosen},
}

@article{GiaquintoZhang,
	title={Bialgebra actions, twists, and universal deformation formulas},
	author={Giaquinto, A. and Zhang, J. J.},
	journal={Journal of Pure and Applied Algebra},
	volume={128},
	number={2},
	pages={133--151},
	year={1998},
	publisher={Elsevier}
}

@book{KarpilovskyContinuation,
  title={Continuation of the Notas de Matem{\`a}tica},
  editor={Nachbin, L.},
  author={Karpilovsky, Gregory},
  isbn={9780444533289},
  year={1993},
  volume={177},
  publisher={North-Holland}
}

@book {Majid_foundations,
	AUTHOR = {Majid, Shahn},
	TITLE = {Foundations of quantum group theory},
	PUBLISHER = {Cambridge University Press, Cambridge},
	YEAR = {1995},
	PAGES = {x+607},
	ISBN = {0-521-46032-8},
	MRCLASS = {17B37 (18D99 81R50)},
	MRNUMBER = {1381692},
	MRREVIEWER = {Dmitri\u i\ I.\ Gurevich},
	DOI = {10.1017/CBO9780511613104},
	URL = {https://doi.org/10.1017/CBO9780511613104},
}

@book {Majid_primer,
    AUTHOR = {Majid, Shahn},
     TITLE = {A quantum groups primer},
    SERIES = {London Mathematical Society Lecture Note Series},
    VOLUME = {292},
 PUBLISHER = {Cambridge University Press, Cambridge},
      YEAR = {2002},
     PAGES = {x+169},
      ISBN = {0-521-01041-1},
   MRCLASS = {17B37 (17B05 58B32 81R50)},
  MRNUMBER = {1904789},
MRREVIEWER = {H.\ H.\ Andersen},
       DOI = {10.1017/CBO9780511549892},
       URL = {https://doi.org/10.1017/CBO9780511549892},
}

@book {PP_QA,
	AUTHOR = {Polishchuk, Alexander and Positselski, Leonid},
	TITLE = {Quadratic algebras},
	SERIES = {University Lecture Series},
	VOLUME = {37},
	PUBLISHER = {American Mathematical Society, Providence, RI},
	YEAR = {2005},
	PAGES = {xii+159},
	ISBN = {0-8218-3834-2},
	MRCLASS = {16S37 (16E30)},
	MRNUMBER = {2177131},
	MRREVIEWER = {Ralf\ Fr\"oberg},
	DOI = {10.1090/ulect/037},
	URL = {https://doi.org/10.1090/ulect/037},
}

@article {Bergman,
	AUTHOR = {Bergman, George M.},
	TITLE = {The diamond lemma for ring theory},
	JOURNAL = {Adv. in Math.},
	FJOURNAL = {Advances in Mathematics},
	VOLUME = {29},
	YEAR = {1978},
	NUMBER = {2},
	PAGES = {178--218},
	ISSN = {0001-8708},
	MRCLASS = {16-02},
	MRNUMBER = {506890},
	MRREVIEWER = {P.\ M.\ Cohn},
	DOI = {10.1016/0001-8708(78)90010-5},
	URL = {https://doi.org/10.1016/0001-8708(78)90010-5},
}

@article {Manin_Koszul,
	AUTHOR = {Manin, Yu.\ I.},
	TITLE = {Some remarks on {K}oszul algebras and quantum groups},
	JOURNAL = {Ann. Inst. Fourier (Grenoble)},
	FJOURNAL = {Universit\'e{} de Grenoble. Annales de l'Institut Fourier},
	VOLUME = {37},
	YEAR = {1987},
	NUMBER = {4},
	PAGES = {191--205},
	ISSN = {0373-0956,1777-5310},
	MRCLASS = {16A24 (17A99 81D05)},
	MRNUMBER = {927397},
	MRREVIEWER = {E.\ J.\ Taft},
	DOI = {10.5802/aif.1117},
	URL = {https://doi.org/10.5802/aif.1117},
}

@book {Manin_CRM,
	AUTHOR = {Manin, Yuri I.},
	TITLE = {Quantum groups and noncommutative geometry},
	SERIES = {CRM Short Courses},
	%EDITION = {Second},
	NOTE = {With a contribution by Theo Raedschelders and Michel Van den
	Bergh},
	PUBLISHER = {Centre de Recherches Math\'ematiques, [Montreal], QC;
	Springer, Cham},
	YEAR = {2018},
	PAGES = {vii+125},
	ISBN = {978-3-319-97986-1},
	MRCLASS = {20G42 (16S37 16S38 16T05 16T25 18Dxx 58B32 58B34)},
	MRNUMBER = {3839605},
	MRREVIEWER = {Liyu\ Liu},
	DOI = {10.1007/978-3-319-97987-8},
	URL = {https://doi.org/10.1007/978-3-319-97987-8},
}

@article {Stafford_vandenBergh,
	AUTHOR = {Stafford, J. T. and van den Bergh, M.},
	TITLE = {Noncommutative curves and noncommutative surfaces},
	JOURNAL = {Bull. Amer. Math. Soc. (N.S.)},
	FJOURNAL = {American Mathematical Society. Bulletin. New Series},
	VOLUME = {38},
	YEAR = {2001},
	NUMBER = {2},
	PAGES = {171--216},
	ISSN = {0273-0979,1088-9485},
	MRCLASS = {16S38 (14A22 16P90 16W50 18E15)},
	MRNUMBER = {1816070},
	MRREVIEWER = {Darin\ R.\ Stephenson},
	DOI = {10.1090/S0273-0979-01-00894-1},
	URL = {https://doi.org/10.1090/S0273-0979-01-00894-1},
}

@article{Shelton_Tingey,
	title = {On Koszul Algebras and a New Construction of Artin–Schelter Regular Algebras},
	journal = {Journal of Algebra},
	volume = {241},
	number = {2},
	pages = {789-798},
	year = {2001},
	issn = {0021-8693},
	doi = {https://doi.org/10.1006/jabr.2001.8781},
	url = {https://www.sciencedirect.com/science/article/pii/S0021869301987812},
	author = {Brad Shelton and Craig Tingey}
}

@misc{MPS,
	title={Creating quantum projective spaces by deforming q-symmetric algebras}, 
	author={Mykola Matviichuk and Brent Pym and Travis Schedler},
	year={2024},
	eprint={2411.10425},
	archivePrefix={arXiv},
	primaryClass={math.QA},
	url={https://arxiv.org/abs/2411.10425}, 
}

@article {torsors,
	AUTHOR = {Guillot, Pierre and Kassel, Christian and Masuoka, Akira},
	TITLE = {Twisting algebras using non-commutative torsors: explicit
	computations},
	JOURNAL = {Math. Z.},
	FJOURNAL = {Mathematische Zeitschrift},
	VOLUME = {271},
	YEAR = {2012},
	NUMBER = {3-4},
	PAGES = {789--818},
	ISSN = {0025-5874,1432-1823},
	MRCLASS = {16T15},
	MRNUMBER = {2945585},
	MRREVIEWER = {Zheng\ Ming\ Jiao},
	DOI = {10.1007/s00209-011-0891-x},
	URL = {https://doi.org/10.1007/s00209-011-0891-x},
}

@article {Jones-Healey,
	AUTHOR = {Jones-Healey, Edward},
	TITLE = {Drinfeld twists of {K}oszul algebras},
	JOURNAL = {Comm. Algebra},
	FJOURNAL = {Communications in Algebra},
	VOLUME = {52},
	YEAR = {2024},
	NUMBER = {8},
	PAGES = {3406--3418},
	ISSN = {0092-7872,1532-4125},
	MRCLASS = {16T05 (16S37 17B37)},
	MRNUMBER = {4748974},
	DOI = {10.1080/00927872.2024.2318658},
	URL = {https://doi.org/10.1080/00927872.2024.2318658},
}
	
\end{document}